\documentclass{amsart}

\usepackage{amsfonts}
\usepackage{amssymb}
\usepackage[mathscr]{eucal}

\newtheorem{theorem}{Theorem}
\newtheorem{lemma}{Lemma}

\def\re{\mathop{\mathfrak{Re}}}
\def\im{\mathop{\mathfrak{Im}}}

\begin{document}

\title{On 4-dimensional, conformally flat, \newline
almost {\rm$\varepsilon$}-K\"ahlerian manifolds}

\author{Karina Olszak and Zbigniew Olszak}

\dedicatory{Dedicated to the memory of Jerzy Julian Konderak}

\date{\today}



\address{Institute of Mathematics and Computer Science, Wroc{\l}aw University of Technology, Wybrze\.ze Wyspia\'nskiego 27, 50-370 Wroc{\l}aw, Poland}

\email{karina.olszak@pwr.wroc.pl, zbigniew.olszak@pwr.wroc.pl}

\begin{abstract}
The local structure of 4-dimensional, conformally flat, almost $\varepsilon$-K\"ahlerian (i.e., almost pseudo-K\"ahlerian and almost para-K\"ahlerian) manifolds is characterized with the help of left-regular and right-regular paraquaternionic functions. Examples of such structures are discussed. 
\end{abstract}

\maketitle

\begin{center}
{\sc 1. Introduction} 
\end{center}

\smallskip
Blair \cite{B} has applied the quaternionic analysis to prove that almost K\"ahler\-ian  manifolds of dimension 4 cannot be of constant non-zero curvature. Recently, Kr\'olikowski \cite{K1,K2,K3} has developed the ideas of applications of quaternionic analysis during investigations of conformally flat almost K\"ahlerian structures. 

In the presented paper, applying the theory of regular paraquaternionic functions introduced by Pogoruy and Rodr\'iguez-Dagnino \cite{PRD}, we propose paraquaternionic analogy of these ideas applied to conformally flat almost pseudo-K\"ahlerian as well as almost para-K\"ahlerian manifolds.

\bigskip
\begin{center}
{\sc 2. Regular paraquaternionic functions} 
\end{center}

\smallskip
Let $\mathbb A$ be the paraquaternion algebra over the real field $\mathbb R$ (which is also called the split-quaternion or coquaternion algebra); cf.\ \cite{C,MS,PRD,R}, etc. 

Elements of $\mathbb A$ are of the form 
$
  {\mathtt x}=x^0+\sum_{\alpha}x^{\alpha}i_{\alpha},
$
where $\sum_{\alpha}$ denotes the sum with respect to $\alpha=1,2,3$, and $x^0,x^{\alpha}\in\mathbb R$, and $1,i_1,i_2,i_3$ is a basis of $\mathbb A$ such that 
$$
   i_1^2=-1,\ i_2^2=i_3^2=1,\ i_1i_2=-i_2i_1=i_3, \\
   i_2i_3=-i_3i_2=-i_1,\ i_1i_3=-i_3i_1=-i_2.
$$
Thus, the multiplication rule of paraquaternions is as follows
\begin{eqnarray*}
  {\mathtt x}{\mathtt y} &=& \Big(x^0+\sum\nolimits_{\alpha}x^{\alpha}i_{\alpha}\Big)
                     \Big(y^0+\sum\nolimits_{\alpha}y^{\alpha}i_{\alpha}\Big) \\
                 &=& x^0y^0 - x^1y^1 + x^2y^2 + x^3y^3 
                     + (x^0y^1 + x^1y^0 - x^2y^3 + x^3y^2)i_1 \\
                 & & \null+(x^0y^2 - x^1y^3 + x^2y^0 + x^3y^1)i_2 
                     + (x^0y^3 + x^1y^2 - x^2y^1 + x^3y^0)i_3.
\end{eqnarray*}

$\mathbb A$ is an associative and non-commutative abstract algebra, which contains zero divisors, nilpotent  elements, and nontrivial idempotents. The paraquaternionic algebra is isomorphic with the Clifford algebras $\mathcal Cl(1,1)\cong\mathcal Cl(0,2)$. The paraquaternions are most familiar through their isomorphism with $2\times 2$ real matrices.

For a paraquaternion $\mathtt x$, define the real part and the imaginary part of $\mathtt x$ by  
$$
  \re({\mathtt x})=x^0,\quad \im({\mathtt x})=\sum\nolimits_{\alpha}x^{\alpha}i_{\alpha},
$$
respectively. The conjugate of $\mathtt x$ is defined by $\overline{\mathtt x}=\re{{\mathtt x}}-\im({\mathtt x})$, and the norm by
$$
  \sqrt{{\mathtt x}\overline{\mathtt x}} = \sqrt{(x^0)^2+(x^1)^2-(x^2)^2-(x^3)^2}.
$$
When ${\mathtt x}\overline{\mathtt x}<0$, then the value of the norm is a complex number from the upper half of
the complex plane. 

The notions of left-regular and right-regular paraquaternionic functions was introduced by Pogoruy and Rodr\'iguez-Dagnino in \cite{PRD} by analogy to the notion of regular quaternionic functions. For quaternionic analysis, we refer e.g. \cite{BLS,F,S}, etc.  

In the sequel, by $\mathcal F$ we will denote the set of all real-differentiable functions $f\colon U\to\mathbb A$, where $U$ is a domain in $\mathbb R^4$.

Consider the following two Cauchy-Fueter type operators $D_L$ and $D_R$ defined on functions $f\in\mathcal F$ by the formulas (we apply denotations different from those in \cite{PRD}) 
$$
  D_Lf = \partial_0f + \sum\nolimits_{\alpha} i_{\alpha}(\partial_{\alpha}f), \quad
  D_Rf = \partial_0f + \sum\nolimits_{\alpha} (\partial_{\alpha}f) i_{\alpha},
$$
where it is supposed that $\partial_j={\partial}/{\partial_j}$. Writing $f$ in the following form 
$
  f=f^0+\sum_{\alpha} f^{\alpha}i_{\alpha},
$
where $f^0,f^{\alpha}\colon U\to\mathbb R$, by straighforward computation, we claim for $D_Lf$,
\begin{eqnarray}
\label{DLfexpl}
 \qquad D_Lf &=& \partial_0f^0-\partial_1f^1+\partial_2f^2+\partial_3f^3 
                 +\big(\partial_1f^0+\partial_0f^1+\partial_3f^2-\partial_2f^3\big)i_1 \\
       & & \null+\big(\partial_2f^0+\partial_3f^1+\partial_0f^2-\partial_1f^3\big)i_2 
            +\big(\partial_3f^0-\partial_2f^1+\partial_1f^2+\partial_0f^3\big)i_3, \nonumber
\end{eqnarray}
and for $D_Rf$, 
\begin{eqnarray}
\label{DRfexpl}
 \qquad D_Rf &=& \partial_0f^0-\partial_1f^1+\partial_2f^2+\partial_3f^3 
                 +\big(\partial_1f^0+\partial_0f^1-\partial_3f^2+\partial_2f^3\big)i_1 \\
       & & \null+\big(\partial_2f^0-\partial_3f^1+\partial_0f^2+\partial_1f^3\big)i_2 
           +\big(\partial_3f^0+\partial_2f^1-\partial_1f^2+\partial_0f^3\big)i_3. \nonumber
\end{eqnarray}

From (\ref{DLfexpl}) and (\ref{DRfexpl}), one notes that it always holds $\re(D_Lf)=\re(D_Rf)$.

The operators $D_L$ and $D_R$ are completely different. In the first lemma, we describe the paraquaternionic functions for which they give the same efects. 

\begin{lemma}
For a paraquaternionic function $f=f^0+\sum_{\alpha}f^{\alpha}i_{\alpha}$, the condition $D_Lf=D_Rf$ holds if and only if $f^{\alpha}=\partial_{\alpha}F$ locally (i.e., in a neighborhood of an arbitrary point), where $F$ is a certain paraquaternionic function. 
\end{lemma}

\begin{proof}
Using (\ref{DLfexpl}) and (\ref{DRfexpl}), we find 
$$
  D_Lf-D_Rf = 2\big(\partial_3f^2-\partial_2f^3\big)i_1 
              + 2\big(\partial_3f^1-\partial_1f^3\big)i_2 
              + 2\big(\partial_1f^2-\partial_2f^1\big)i_3. 
$$
Hence, the condition $D_Lf=D_Rf$ holds if and only if 
$$
  \partial_3f^2-\partial_2f^3=0, \quad 
  \partial_3f^1-\partial_1f^3=0, \quad 
  \partial_1f^2-\partial_2f^1=0. 
$$
When fixing $x^0$, the above condtion says that the 1-form $\omega=f^1dx^1+f^2dx^2+f^3dx^3$ is closed. By the famous Poincare lemma, the form $\omega$ is locally a boundary form, say $\omega=dF$ for a certain paraquaternionic function $F$. This completes the proof. 
\end{proof}

\begin{lemma}
The operators $D_L$ and $D_R$ commute, that is, for any $f\in\mathcal F$, 
\begin{equation}
\label{commDLDR}
  D_LD_Rf-D_RD_Lf=0.
\end{equation}
\end{lemma}

\begin{proof}
Since the partial derivatives of functions $f^i$, $ 0\leqslant i\leqslant 3$, commute and the paraquaternionic multiplication is associative, the conclusion follows obviously.
\end{proof}

Let $f\in\mathcal F$. $f$ is called left-regular if it satisfies the Cauchy-Riemann-Feuter type equation $D_Lf=0$. And similarly, $f$ is called right-regular if it satisfies the Cauchy-Riemann-Feuter type equation $D_Rf=0$. 

Having (\ref{DLfexpl}) and (\ref{DRfexpl}), we see that $f$ is left-regular if and only if
\begin{equation}
\label{eqDleft}
  \left\{
  \begin{array}{l}
  \partial_0f^0-\partial_1f^1+\partial_2f^2+\partial_3f^3=0,\\[+3pt]
  \partial_1f^0+\partial_0f^1+\partial_3f^2-\partial_2f^3=0,\\[+3pt]
  \partial_2f^0+\partial_3f^1+\partial_0f^2-\partial_1f^3=0,\\[+3pt]
  \partial_3f^0-\partial_2f^1+\partial_1f^2+\partial_0f^3=0,
  \end{array}
  \right.
\end{equation}
and $f$ is right-regular if and only if
\begin{equation}
\label{eqDright}
  \left\{
  \begin{array}{l}
  \partial_0f^0-\partial_1f^1+\partial_2f^2+\partial_3f^3=0,\\[+3pt]
  \partial_1f^0+\partial_0f^1-\partial_3f^2+\partial_2f^3=0,\\[+3pt]
  \partial_2f^0-\partial_3f^1+\partial_0f^2+\partial_1f^3=0,\\[+3pt]
  \partial_3f^0+\partial_2f^1-\partial_1f^2+\partial_0f^3=0.
  \end{array}
  \right.
\end{equation}

\begin{lemma}
(a) If $f$ is a left-regular paraquaternionic function, then $h=D_Rf$ is also left-regular paraquaternionic and $\re(h)=0$. (b) If $f$ is a right-regular paraquaternionic function, then $h=D_Lf$ is also right-regular paraquaternionic and $\re(h)=0$. 
\end{lemma}

\begin{proof}
(a) Since $D_Lf=0$, by (\ref{commDLDR}), we have $D_Lh=D_LD_Rf=D_RD_Lf=0$. Moreover, using formula (\ref{DRfexpl}), we find 
\begin{equation}
\label{reh1}
  \re(h)=\re(D_Rf)=\partial_0f^0-\partial_1f^1+\partial_2f^2+\partial_3f^3.
\end{equation}
But, by (\ref{DLfexpl}), the right hand side of (\ref{reh1}) is the same as $\re(D_Lf)$, which is equal to 0 since $D_Lf=0$.

(b) Since $D_Rf=0$, by (\ref{commDLDR}), we have $D_Rh=D_RD_Lf=D_LD_Rf=0$. Moreover, using formula (\ref{DLfexpl}), we find 
\begin{equation}
\label{reh2}
  \re(h)=\re(D_Lf)=\partial_0f^0-\partial_1f^1+\partial_2f^2+\partial_3f^3.
\end{equation}
But, by (\ref{DRfexpl}), the right hand side of (\ref{reh2}) is the same as $\re(D_Rf)$, which is equal to 0 since $D_Rf=0$.
\end{proof}

The simplest examples of left-regular and right-regular paraquaternionic functions are the so-called paraquaternionic Fueter polynomials
$$
  \zeta_{\alpha}(x)=x^{\alpha}-x^0i_{\alpha},\quad x\in\mathbb R^4.
$$

From the result achieved in \cite{PRD} it follows that a paraquaternionic function is left-regular on $U$ if and only if it can be expanded, in a neighborhood of an arbitrary point $x_0\in U$, into a series 
$$
  f(x) = f(x_0) + 
  \sum_{n=1}^{\infty}\Big(\sum_{\alpha_1,\alpha_2,\ldots,\alpha_n=1}^3
           \zeta_{\alpha_1}(x-x_0)
           \zeta_{\alpha_2}(x-x_0)\cdots
           \zeta_{\alpha_n}(x-x_0)
           a_{\alpha_1\alpha_2\cdots\alpha_n}\Big),
$$
where $a_{\alpha_1\alpha_2\cdots\alpha_n}\in\mathbb A$ for any $1\leqslant\alpha_i\leqslant3$, $1\leqslant i\leqslant n$ and $n\in\mathbb N$. 

Hence, we claim that the paraquaternionic functions of the form 
$$
  f(x) = 
  \sum_{k=1}^n\Big(\sum_{\alpha_1,\alpha_2,\ldots,\alpha_n=1}^3
           \zeta_{\alpha_1}(x)\zeta_{\alpha_2}(x)\cdots\zeta_{\alpha_k}(x)
           a_{\alpha_1\alpha_2\cdots\alpha_k}\Big),
$$
where $a_{\alpha_1\alpha_2\cdots\alpha_k}\in\mathbb A$, are left-regular. 

In a similar way, we can claim tha the paraquaternionic functions of the form 
$$
  f(x) = 
  \sum_{k=1}^n \Big(\sum_{\alpha_1,\alpha_2,\ldots,\alpha_n=1}^3
   a_{\alpha_1\alpha_2\cdots\alpha_k}
  \zeta_{\alpha_1}(x)\zeta_{\alpha_2}(x)\cdots\zeta_{\alpha_k}(x)\Big),
$$
where $a_{\alpha_1\alpha_2\cdots\alpha_k}\in\mathbb A$, are right-regular. 

\bigskip
\begin{center}
{\sc 3. Conformal flatness of 4-dimensional almost {\rm$\varepsilon$}-K\"ahlerian manifolds}
\end{center}

\smallskip
Let  $M$ a $4$-dimensional, connected, differentiable manifold. Assume that $M$ is endowed with a pseudo-Riemannian metric $g$ of Norden signature $(--++)$ and a $(1,1)$-tensor field $J$ such 
$$
  J^2=\varepsilon I,\quad g(JX,JY)=-\varepsilon g(X,Y)
$$
for any $X,Y\in\mathfrak X(M)$, where $\varepsilon$ is a constant equal to $+1$ or $-1$. 

If $\varepsilon=-1$, then the manifold $M(J,g)$ is an almost pseudo-Hermitian or indefinite almost Hermitian (cf. \cite{MY,N}, etc.). If $\varepsilon=+1$, then the manifold $M(J,g)$ is an almost para-Hermitian (cf. \cite{BFFV,CFG}, etc.). For simplicity, we will say that $M(J,g)$ or just $M$ is an almost $\varepsilon$-Hermitian manifold.

For an almost $\varepsilon$-Hermitian manifold $M(J,g)$, let $\varOmega$ be the skew-symmetric $(0,2)$-tensor file defined by $g(X,JY)=\varOmega(X,Y)$. $\varOmega$ is an almost symplectic form on $M$ and it is called the fundamental form of this manifold.

Let $M(J,g)$ be a 4-dimensional almost $\varepsilon$-Hermitian manifold. Assume additionally that the manifold $M(J,g)$ is (locally) conformally flat. This means that for any point $p\in M$ there exist an open neighborhood $U\subset M$ of the point $p$ and a positive function $h\colon U\to\mathbb R$ such that $g=hG$, where $G$ is a flat pseudo-Riemannian metric. We restrict our consideration to the subset $U$. Without lose of generality, we assume that $U$ is a coordinate neighborhood with coordinates $(x^0,x^1,x^2,x^3)$ (it will be useful to number the indices from 0 to 3) and the metric $G$ has components 
\begin{equation}
\label{compG}
  G_{11}=G_{22}=-1,\quad G_{33}=G_{44}=1,\quad\mbox{and} \ G_{ij}=0 \ \mbox{otherwise}. 
\end{equation}
Thus, the metric $g$ has the following components 
$$
  g_{11}=g_{22}=-h,\quad g_{33}=g_{44}=h,\quad\mbox{and} \ g_{ij}=0 \ \mbox{otherwise}. 
$$
Let $J^i_j$ be the components of the tensor field $J$. Now, we compute the components $J^i_j$ having the equalities
$$
  \sum\limits_s J^i_s J^s_j=\varepsilon\delta^i_j,\quad 
  \sum\limits_s g_{is} J^s_j+\sum\limits_s g_{js} J^s_i=0.
$$
By some standard computations, we obtain the following two possibilities for the matrix $[J^i_j]$,
\begin{equation}
\label{compJ1}
  [J^i_j] = \left[
  \begin{array}{@{\extracolsep{1mm}}cccc}
    0 &  a &  b &  c \\
   -a &  0 & -c &  b \\
    b & -c &  0 & -a \\
    c &  b &  a &  0
  \end{array}
  \right],
\end{equation}
or
\begin{equation}
\label{compJ2}
  [J^i_j] = \left[
  \begin{array}{@{\extracolsep{1mm}}cccc}
    0 &  a &  b &  c \\
   -a &  0 &  c & -b \\
    b &  c &  0 &  a \\
    c & -b & -a &  0
  \end{array}
  \right],
\end{equation}
$a,b,c$ being functions such that $a^2-b^2-c^2=-\varepsilon$. Consequently, the matrices of the components $\varOmega_{ij}$ of the fundamental form are of the form
\begin{equation}
\label{compOmega1}
  [\varOmega_{ij}] = \left[
  \begin{array}{@{\extracolsep{1mm}}cccc}
      0 & -f^1 & -f^2 & -f^3 \\
    f^1 &    0 &  f^3 & -f^2 \\
    f^2 & -f^3 &    0 & -f^1 \\
    f^3 &  f^2 &  f^1 &    0
  \end{array}
  \right],
\end{equation}
or
\begin{equation}
\label{compOmega2}
  [\varOmega_{ij}] = \left[
  \begin{array}{@{\extracolsep{1mm}}cccc}
      0 & -f^1 & -f^2 & -f^3 \\
    f^1 &    0 & -f^3 &  f^2 \\
    f^2 &  f^3 &    0 &  f^1 \\
    f^3 & -f^2 & -f^1 &    0
  \end{array}
  \right],
\end{equation}
where we supposed $f^1=ha$, $f^2=hb$ and $f^3=hc$. The functions $f^1$, $f^2$ and $f^3$ satisfy the condition $(f^1)^2-(f^2)^2-(f^3)^2 = -\varepsilon h^2$. 

The coordinate neighborhood $(U,(x^0,x^1,x^2,x^3))$ constructed in the above will be said to be the canonical coordinate neighborhood of a conformally flat 4-dimen\-sion\-al almost $\varepsilon$-Hermitian. 

Assume that for an almost $\varepsilon$-Hermitian manifold $M(J,g)$, the fundamental form $\varOmega$ is closed ($d\varOmega=0$, that is, $\varOmega$ is a symplectic form). If $\varepsilon=-1$, then the manifold $M(J,g)$ is called almost pseudo-K\"ahlerian or indefinite almost K\"ahlerian (cf.\ e.g.\ \cite{DD,SY}). If $\varepsilon=+1$, then the manifold $M(J,g)$ is called almost para-K\"ahlerian (cf.\ e.g.\ \cite{IZ}). For simplicity, we will say that $M(J,g)$ or just $M$ is an almost $\varepsilon$-K\"ahlerian manifold.

\begin{theorem}
\label{theo}
A 4-dimensional, conformally flat, almost $\varepsilon$-Hermitian manifold \linebreak $M(J,g)$ is almost $\varepsilon$-K\"ahlerian if and only if for any point $p\in M$ there is a canonical coordinate neighborhood $(U,(x^0,x^1,x^2,x^3))$, $p\in U$, on which there exist functions $f^1,f^2,f^3$ and a positive function $h$ such that \\
(i) the functions $f^1,f^2,f^3,h$ are related by the condition
\begin{equation}
\label{funch}
  (f^1)^2-(f^2)^2-(f^3)^2 = -\varepsilon h^2, 
\end{equation}
and the paraquaternionic function $f=f_1i_1+f^2i_2+f^3i_3$ is left-regular or right-regular; \\
(ii) $g=hG$, $G$ being the standard flat Norden type metric (cf. (\ref{compG})); \\
(iii) if $f$ is left-regular, then $J$ is given by (\ref{compJ1}), and if $f$ is right-regular, then $J$ is give by (\ref{compJ2}), where in the both cases, $a,b,c$ are the functions given by 
$$
  a=\frac{f^1}{h},\quad b=\frac{f^2}{h},\quad c=\frac{f^3}{h}.
$$
\end{theorem}

\begin{proof}
Let $M(J,g)$ be a  4-dimensional, conformally flat, almost $\varepsilon$-Hermitian manifold. In a canonical coordinate neighborhood, the components of the exterior derivative $d\varOmega$ are of the form
\begin{equation}
\label{compdOm}
  (d\varOmega)_{ijk} = \frac13\big(
  \partial_i\varOmega_{jk}+\partial_j\varOmega_{ki}+\partial_k\varOmega_{ij}\big).
\end{equation}

Let us assume that $M(J,g)$ is almost $\varepsilon$-K\"ahlerian. Since $d\varOmega=0$, from (\ref{compdOm}), we obtain 
\begin{equation}
\label{dOm}
  \partial_i\varOmega_{jk}+\partial_j\varOmega_{ki}+\partial_k\varOmega_{ij}=0.
\end{equation}

In the case when $\varOmega$ is of the form (\ref{compOmega1}), the relations (\ref{dOm}) can be written equivalently as the following system
$$
  \left\{
  \begin{array}{r}
  -\partial_1f^1+\partial_2f^2+\partial_3f^3=0,\\[+3pt]
   \partial_0f^1+\partial_3f^2-\partial_2f^3=0,\\[+3pt]
   \partial_3f^1+\partial_0f^2-\partial_1f^3=0,\\[+3pt]
  -\partial_2f^1+\partial_1f^2+\partial_0f^3=0.
  \end{array}
  \right.
$$
Comparing this system with (\ref{eqDleft}) we claim that the paraquaternionic function $f=f^1i_1+f^2i_2+f^3i_3$ is left-regular.

In the case when $\varOmega$ is of the form (\ref{compOmega2}), the relations (\ref{dOm}) can be written equivalently as the following system
$$
  \left\{
  \begin{array}{r}
  -\partial_1f^1+\partial_2f^2+\partial_3f^3=0,\\[+3pt]
   \partial_0f^1-\partial_3f^2+\partial_2f^3=0,\\[+3pt]
  -\partial_3f^1+\partial_0f^2+\partial_1f^3=0,\\[+3pt]
   \partial_2f^1-\partial_1f^2+\partial_0f^3=0.
  \end{array}
  \right.
$$
Comparing this system with (\ref{eqDright}), we claim that the paraquaternionic function $f=f^1i_1+f^2i_2+f^3i_3$ is right-regular.

Recall that the function $h$ deforming the metric $g$ to a flat metric is related to the functions $f^1,f^2,f^3$ by (\ref{funch}).

The converse is clear.
\end{proof}

In view of Theorem \ref{theo}, to construct examples of 4-dimensional conformally flat almost $\varepsilon$-K\"ahlerian structures it is necessary and sufficient to have left-regular and right-regular paraquaternionic functions $f=f^1i_1+f^2i_2+f^3i_3$ for which $(f^1)^2-(f^2)^2-(f^3)^2\not=0$. Using the results from the previous section, one can find such functions. We are going to state only two concrete examples of such functions. Namely, \\
(a) the function 
\begin{eqnarray*}
  f(x) &=& (\zeta_1(x)+\zeta_2x)+\zeta_3(x))(\null-i_2+i_3) \\
       &=& 2x^0i_1+(x^0-x^1-x^2-x^3)i_2+(x^0+x^1+x^2+x^3)i_3,
\end{eqnarray*}
which is left-regular and realizes the desired conditions on any domain in $\mathbb R^4$ on which $|x^0|\neq|x^1+x^2+x^3|$; and similarly \\
(b) the function 
\begin{eqnarray*}
  f(x) &=& (i_2-i_3)(\zeta_1(x)+\zeta_2x)+\zeta_3(x)) \\
       &=& 2x^0i_1+(x^0+x^1+x^2+x^3)i_2+(x^0-x^1-x^2-x^3)i_3,
\end{eqnarray*}
which is right-regular and realizes the desired conditions on any domain in $\mathbb R^4$ on which $|x^0|\neq|x^1+x^2+x^3|$. Having these functions, one can write down the almost $\varepsilon$-K\"ahlerian structures explicitly. 
  

\end{document}